\documentclass[a4paper,11pt,twoside,notitlepage,reqno]{amsart}
\usepackage{amsmath,amssymb,amsbsy,amsthm}
\usepackage[hmargin=1.4in,vmargin=1.4in]{geometry}
\usepackage{url}

\usepackage{tikz}
\usepackage[linktocpage]{hyperref}
\hypersetup{
    colorlinks=true,        % false: boxed links; true: colored links
    linkcolor=red,          % color of internal links (change box color with linkbordercolor)
    citecolor=red,         % color of links to bibliography
    filecolor=magenta,      % color of file links
    urlcolor=cyan           % color of external links
}

\addtocontents{toc}{\protect\setcounter{tocdepth}{1}}

\usepackage{eucal}

\newcommand{\F}{ \mathbb{F} }

\theoremstyle{plain}
	\newtheorem{thm}{Theorem}
	
		\numberwithin{thm}{section}
	\newtheorem{lemma}[thm]{Lemma}
	\newtheorem{prop}[thm]{Proposition}

    \newtheorem{prob}[thm]{Problem}

	\newtheorem*{thm*}{Theorem}
	\newtheorem*{lemma*}{Lemma}
	\newtheorem*{prop*}{Proposition}
	\newtheorem*{cor*}{Corollary}
	\newtheorem*{conj*}{Conjecture}

\theoremstyle{definition}
	\newtheorem{example}[thm]{Example}
	\newtheorem*{example*}{Example}

	\newtheorem{remark}[thm]{Remark}
\begin{document}

\title[Companion weakly periodic matrices over fields]{Companion weakly periodic matrices over \\ finite and countable fields}

\author[P. Danchev]{Peter Danchev}
\address{Institute of Mathematics and Informatics, Bulgarian Academy of Sciences, ``Acad. G. Bonchev" str., bl. 8, 1113 Sofia, Bulgaria}
\email{danchev@math.bas.bg; pvdanchev@yahoo.com}

\author[A. Pojar]{Andrada Pojar}
\address{The Technical University of Cluj-Napoca, Departament of Mathematics, Str. Memorandumului 28, 400114, Cluj-Napoca, Romania}
\email{andradapojar@gmail.com}


\begin{abstract} We explore the situation where all companion $n\times n$ matrices over a field $\F$ are weakly periodic of index of nilpotence $2$ and prove that this can be happen uniquely when $\F$ is a countable field of positive characteristic, which is an algebraic extension of its minimal simple (finite) subfield, with all subfields of order greater than $n$. In particular, in the commuting case, we show even that $\F$ is a finite field of order greater than $n$.

Our obtained results somewhat generalize those obtained by Breaz-Modoi in Lin. Algebra \& Appl. (2016).
\end{abstract}

\maketitle

\section{Introduction and Principalities}

Let $\F$ be a field and $n$ an arbitrary non-negative integer, say $n\geq 1$. We denote by $M_n(\F)$ the matrix ring consisting of all squared $n\times n$ matrices over $\F$. For a non-negative integer $m\geq 2$, we denote by $\tau$ the cycle $\tau_m=(1 ~ 2 ~ \ldots ~ m) \in S_m$, and by $P_{\tau_m}$ the $m\times m$ permutation matrix with only $0$s and $1$s over $\F$, that is, $P_{\tau_m}=(a_{ij})_{1\leq i,j \leq m}$, where $(a_{ij})=1$ if $j=\tau_m(i)$, and $(a_{ij})=0$ if $j\neq\tau_m(i)$.

The matrix $M\in M_n(\F)$ is called {\it potent} if there exists an integer $k\geq 1$ such that $M^{k+1}=M$. So, let $ST_n$ the set of traces of all potent companion matrices $C\in M_n(\F)$. By a root of unity, we denote a root of the polynomial over $\F$ which is of the form $g=X^i-1$, where $i$ is an integer with $i \geq 2$. Thus, let $SR_n$ be the set of sums of at most $n$-th roots of unity which are not necessarily distinct.

Likewise, we denote by $L_{m,n}$ the set of polynomials having degree at most $m$ and with non-negative integer multiples of unity coefficients such that their sum is not exceeding $n$. We also denote $W_m=\cup_{f \in L_{m,n}} Spec(f(P_{\tau_m}))$, and the symbol $d(m)$ stands for the number of non-negative divisors of $m$.

Let $R$ be a ring. An element $x\in R$ is said to be {\it potent} if there exists a non-negative integer $q\geq 1$ with $x^{q+1}=a$, and $x\in R$ is said to be {\it weakly periodic with nilpotence index $2$}, provided $x$ is the sum of a potent and a square-zero nilpotent of $R$. Furthermore, we shall say that the ring $R$ is {\it weakly periodic with nilpotence index $2$}, provided every element of $R$ is weakly periodic with nilpotence index $2$.

Historically, the concept of weak periodicity arisen quite normally in the existing on the subject literature. In fact, it was showed in \cite{CD} that an element $x$ of a ring $R$ is {\it periodic}, i.e., $x^n=x^m$ for some two different positive integers $m,n$, if and only if $x$ can be written as a sum of a potent element and a nilpotent element which commute each other. Thus, by removing the "commuting property", it is rather natural to consider the sum of such two elements. In addition, a ring $R$ is called {\it weakly periodic} if all its elements are weakly periodic.

On the other hand, Diesl defined in \cite{D} the notion of a {\it nil-clean} ring $R$ as the ring for which, for every $a\in R$, there are an idempotent $e$ and a nilpotent $b$ such that $a=e+b$. Moreover, Ye introduced in \cite{Y} the notion of {\it semi-clean} ring $R$ as the ring for which, for each $a\in R$, there are a potent $c$ and a unit $u$ such that $a=c+u$. Henceforth, it is pretty obvious that the weakly periodic rings are situated between nil-clean and semi-clean rings. So, they really need a detailed exploration to which is devoted the present article being our basic motivation. Our major purpose is to decide what is the power structure of the base field, provided that all companion matrices over it are either commuting weakly periodic with index of nilpotence $2$ or just weakly periodic with index of nilpotence $2$, respectively. In what follows, we shall show that such a field is either finite or countably infinite, respectively.

\section{Main Results}

We divide our basic results into two subsections as follows:

\subsection{The general case}

We begin our work with the following plain but useful technicality.

\begin{lemma}\label{SRST}
Let $n\geq 1$ be an integer. Then, the inclusion $ST_n\subseteq SR_n$ holds.
\end{lemma}

\begin{proof}
Let $t\in ST_n$ be the trace of a potent companion matrix $C$ such that $C^{m+1}=C$ for some non-negative integer $m\geq 1$. Let $\lambda$ be an eigenvalue of $C$. Then, $\lambda^{m+1}=\lambda$. It follows that either $\lambda=0$ or $\lambda^{m}=1$. Suppose $k$ is the number of distinct non-zero distinct eigenvalues of $C$, say $\lambda_1,\lambda_2,\ldots,\lambda_k$. Thus, it follows that there exist non-negative integers $\alpha_1, \alpha_2,\ldots, \alpha_k, \alpha_{k+1}$ with sum $n$, and non-zero integers $\alpha_1,\alpha_2,\ldots,\alpha_k$ such that
$$t=\alpha_1\lambda_1+\alpha_2\lambda_2+\ldots+\alpha_k\lambda_k+\alpha_{k+1}\cdot 0.$$
Therefore,
$$t=\alpha_1\lambda_1+\alpha_2\lambda_2+\ldots+\alpha_k\lambda_k,$$ such that the inequality $\alpha_1+\alpha_2+\ldots+\alpha_k\leq n$ is fulfilled. Hence, $t$ is a sum of at most $n$ $m$-th roots of unity. Consequently, $t\in SR_n$, as required.
\end{proof}

We continue with a few pivotal for our work statements.

\begin{prop}\label{wp2compmat}
Let $n\geq 1$ be a non-negative integer and $C \in M_n(\F)$ a companion matrix over a field $\F$. Then the following two claims are true:
\begin{enumerate}
    \item If $C$ is weakly periodic with nilpotence index $2$, then $\mbox{trace}~C\in SR_n$.
    \item If $\mbox{trace}~C \in ST_n $, then $C$ is weakly periodic with nilpotence index $2$.
\end{enumerate}
\end{prop}

\begin{proof}

(1) Assume $C=M+N$ with $M$ potent and $N$ a square-zero nilpotent such that $\mbox{trace}C \notin SR_n$. So, it is obvious that $\mbox{trace}C \notin ST_n$. Since $\mbox{trace}~N=0$, it follows that $\mbox{trace}~C=\mbox{trace}~M$. Thus, one deduces that $M$ is not similar to a companion matrix. We shall now consider the next two cases:

\noindent{\bf Case 1}: $M$ is similar to a direct sum of two companion matrices $C_1$ of size $l$ and $C_2$ of size $n-l$, so that there exists an invertible matrix of size $n$ over $\F$ with $M=U(C_1\oplus C_2)U^{-1}$. Hence, $\mbox{trace}~M=\mbox{trace}~C=\mbox{trace}~C_1+\mbox{trace}~C_2$. Since the direct sum $C_1\oplus C_2$ is obviously potent, one sees that $\mbox{trace}~C_1\in ST_l$ and $\mbox{trace}~C_2\in ST_{n-l}$. By virtue of Lemma \ref{SRST}, one gets that $\mbox{trace}~C_1\in SR_l$ and $\mbox{trace}~C_2\in SR_{n-l}$. Therefore,       $\mbox{trace}~C_1+\mbox{trace}~C_2\in SR_n$, and hence $\mbox{trace}~M\in SR_n$. But $\mbox{trace}~C=\mbox{trace}~M$, and so we obtain that $\mbox{trace}~C\in SR_n$.

\noindent{\bf Case 2}: $M$ is similar to a direct sum of $u$ companion matrices, where $3\leq u\leq n$. It can easily be seen that with induction on $u$ we will have $\mbox{trace}~C\in SR_n$.

(2) Since $\mbox{trace}~C\in ST_n$, there is a potent companion matrix $C'$ such that $\mbox{trace}~C'=\mbox{trace}~C$, whence $(C-C')^2=0$. In conclusion, $C$ is weakly periodic with nilpotence index $2$, as claimed.
\end{proof}

\begin{prop}\label{FinSRn}
Let $n\geq 2$ be an integer. If all $n\times n$ companion matrices are weakly periodic with nilpotence index $2$, then $\F \subseteq SR_n$.
\end{prop}

\begin{proof}
This follows immediately with the aid of Proposition \ref{wp2compmat}.
\end{proof}

\begin{lemma}\label{Wm}
Let $n\geq 2$ be an integer. Then the following containment is valid:
$$\cup_{k \leq n, d(m) \geq k} W_m \subseteq\cup_{m \in \mathbb{N}, m \geq 2} W_m.$$
\end{lemma}

\begin{proof}
This follows directly from the easy fact that, for any non-negative integer $m\geq 2$, if $m$ has at least $k$ divisors, then $m$ has at least one divisor.
\end{proof}

\begin{lemma}\label{SRnWm}
Let $n\geq 1$ be a non-negative integer. Then the following relation holds:
$$SR_n \subseteq\cup_{m \in \mathbb{N}, m \geq 2} W_m.$$
\end{lemma}

\begin{proof} Letting $t\in SR_n$, then there exist non-zero integers $1\leq k\leq n$ and $\alpha_1, \alpha_2,  \ldots \alpha_k$ with sum not exceeding $n$ as well as there are roots of unity $\lambda_1,\lambda_2, ...\lambda_k$ and non-negative integers $m_1, m_2,\ldots, m_k$ such that $\lambda_1^{m_1}=\lambda_2^{m_2}=\ldots=\lambda_k^{m_k}=1$ and such that

$$t=\alpha_1\lambda_1+\alpha_2\lambda_1+\ldots \alpha_k\lambda_k.$$
If we take $m$ to be the common multiple of $m_1,m_2,\ldots,m_k$, then $\lambda_i^m=1$ for every $i \in \{1,2,\ldots,k\}$. Consequently, $k\leq m$ and $$\{\lambda_1,\lambda_2,\ldots,\lambda_k\}\subseteq \{\epsilon^{a_1}+\epsilon^{a_2},\ldots \epsilon^{a_k}\},$$ where $\{a_1,a_2,\ldots,a_k\}\subseteq\{0,1,\ldots, m-1\}$ and $\epsilon$ is so that it generates all $m$-th roots of unity. Without loss of generality, we can assume that $a_1<a_2<\ldots<a_k$. Therefore, $$t=\alpha_1\epsilon^{a_1}+\alpha_2\epsilon^{a_2}+\ldots+\alpha_k\epsilon^{a_k}.$$
Suppose now that $$f=\alpha_1X^{a_1}+\alpha_2X^{a_2}+\ldots+\alpha_kX^{a_k}.$$
Since $\alpha_1, \alpha_2,  \ldots , \alpha_k$ are non-negative integers with sum not exceeding $n$, it will follow that $f\in L_{m,n}$. We also will have the equality $$t=f(\epsilon).$$
Since $\epsilon$ is a primitive $m$-th root of unity, one extracts that $\epsilon$ is an eigenvalue for $P_{\tau_m}$, and hence $f(\epsilon)$ is an eigenvalue of $f(P_{\tau_m})$.

Furthermore, since $m$ is a common multiple $k$ non-zero integers greater than $1, m_1, m_2, \ldots, m_k$, it must be that $m\geq 2$ is a non-negative integer. Now, as $m$ is a non-negative integer, $m$ is a common multiple of number $k$ of its divisors if, and only if, $m$ has at least $k$ divisors and thus, in conjunction with Lemma \ref{Wm}, one derives that
$$SR_n\subseteq\cup_{k \leq n, d(m) \geq k} W_m \subseteq\cup_{m \in \mathbb{N}, m \geq 2} W_m,$$
as asserted.
\end{proof}

\begin{lemma}\label{eq4omega}
Let $m\geq 2$ and $n\geq 1$ be integers. If $\omega \in W_m$, then there exist two integer multiples of unity $u=u(\omega)$ and $\alpha = \alpha(\omega)$ such that $(\omega -\alpha )^m=u$.
\end{lemma}

\begin{proof}
Given $\omega\in W_m$. Then, for every $f\in L_{m,n}$, we have $\omega\in Spec(f(P_{\tau_m})$. It follows now that $\omega$ is a root of the polynomial $det(XI_m-f(P_{\tau_m})$. But as $f\in L_{m,n}$, there exist $1\leq k\leq m$ and non-zero integers $\alpha_1, \alpha_2, \ldots, \alpha_k$ and non-negative pairwise distinct integers $a_1, a_2,\ldots, a_k$ with values at most $n-1$ such that $$f=\alpha_1X^{a_1}+\alpha_2X^{a_2}+\ldots+\alpha_kX^{a_k}.$$
Therefore, $$f(P_{\tau_m})=\alpha_1(P_{\tau_m})^{a_1}+\alpha_2(P_{\tau_m})^{a_2}+\ldots+\alpha_k(P_{\tau_m})^{a_k}.$$
Since $\tau_m$ is a cycle in $S_m$, and $\tau_m=(1 ~ 2 \ldots ~ m)$, it follows that $\tau_m$ is a product of $m-1$ transpositions, and hence $\tau_m$ and $m$ obviously have different parities.

Furthermore, assuming $m$ is even, it follows then that $\tau_m$ is odd, and thus $\tau_m^{a_i}$ and $a_i$ have the same parity for every $i\in \{1,2,\ldots,k\}$. Now, let $(b_{ij})_{1\leq i,j\leq m}=XI_m-f(P_{\tau_m})$ and $\alpha_i=f(0)$. Therefore, $$det(f(P_{\tau_m})-XI_m)=(\alpha_i-XI_m)^m+\sum_{\sigma\neq e,\sigma even}b_{1\sigma(1)}\cdots b_{m\sigma(m)}-\sum_{\sigma odd}b_{1\sigma(1)}\cdots b_{m\sigma(m)}=$$
$$=(\alpha_i-X)^m+(-1)^{a_1+1}\alpha_1^m+(-1)^{a_2+1}\alpha_2^m+\ldots+(-1)^{a_k+1}\alpha_k^m-(-1)^{a_i+1}\alpha_i^m.$$
Consequently, for even $m$, we obtain: $$(\alpha_i-\omega)^m=(-1)^{a_1}\alpha_1^m+(-1)^{a_2}\alpha_2^m+\ldots+(-1)^{a_k}\alpha_k^m-(-1)^{a_i}\alpha_i^m.$$

Assuming that $m$ is now odd, we then infer that $\tau_m$ is even, and so $\tau_m^{a_i}$ is even for each $i\in \{1,2,\ldots,k\}$. Now, let $(b_{ij})_{1\leq i,j\leq m}=f(P_{\tau_m})-XI_m$ and $\alpha_i=f(0)$
Therefore, $$det(f(P_{\tau_m})-XI_m)=(\alpha_i-X)^m+\sum_{\sigma\neq e,\sigma even}b_{1\sigma(1)}\cdots b_{m\sigma(m)}-\sum_{\sigma odd}b_{1\sigma(1)}\cdots b_{m\sigma(m)}=$$
$$=(\alpha_i-X)^m+\alpha_1^m+\alpha_2^m+\ldots+\alpha_k^m-0-\alpha_i^m.$$
Consequently, for odd $m$, we conclude: $$(\omega-\alpha_i)^m=\alpha_1^m+\alpha_2^m+\ldots+\alpha_k^m-\alpha_i^m.$$

Finally, there exist multiple integers of unity, say $\alpha=\alpha_i$ and $u$, such that $(\omega-\alpha)^m=u$, as required.
\end{proof}

The next claim is closely related to assertions from \cite{LAMA}.

\begin{lemma}\label{ADiagIsNonderog}
Every diagonalizable matrix over the finite Galois field $GF(p^l)$ for some $l\geq 1$ with no multiple eigenvalues is a non-derogatory potent matrix.
\end{lemma}

\begin{proof}
Since every element $x\in GF(p^l)$ satisfies the equation $x^{p^l}=x$, it follows that each diagonal over $GF(p^l)$ is potent. Hence, any diagonalizable matrix over $GF(p^l)$ is necessarily potent.

Since the matrix has no multiple eigenvalues, the algebraic multiplicity of every eigenvalue is exactly $1$. And since the matrix is diagonalizable, the geometric multiplicity of any eigenvalue equals to the algebraic multiplicity, so will be the $1$ too. Therefore, the matrix is non-derogatory as well.

In conclusion, the matrix is a non-derogatory potent matrix, as expected.
\end{proof}

The following technical claim from field theory is our key to establish the chief result stated below.

\begin{prop}\label{fields} Each infinite field, which is an algebraic extension of its minimal simple (finite) subfield, is an infinite countable union of its finite subfields, and thus it is countable. In addition, these finite subfields can be taken of the sort GF$(p^{l_i})$ such that GF$(p^{l_i})$ is contained in GF$(p^{l_{i+1}})$ and $l_i$ divides $l_{i+1}$ for each $i \geq 1$.
\end{prop}

\begin{proof} Given $\F$ is an infinite field of characteristic $p>0$, which is an algebraic extension of its simple subfield ${\F}_p$ of $p$ elements, and $u$ is an arbitrary element of $\F$. Hence, each ${\F}_p(u)$ which means the extension of $\F_p$ generated by $u$, is a finite extension of ${\F}_p$ and so it is a finite subfield of $\F$. Precisely, all finite subfields of $\F$ are of this type, because for the finite extensions of ${\F}_p$ is valid the classical theorem for the "primitive element". We also know that $\F_p(u)\cong F_p[X]/(g_{u}(X))$, where $X$ is a transcendental element over $\F_p$, and $g_{u}(X)$ is the minimal polynomial of $u$ over $\F_p$. To ensure an uniqueness of $g_{u}(X)$, as usual we will assume that it is normed, that is, its chief coefficient is exactly $1$. Also, the degree $[{\F_p}(u):{\F}_p]$ of the extension ${\F_p}(u)/{\F}_p$ coincides with the degree of the polynomial $g_{u}(X)$.

Furthermore, the extensions ${\F_p}(u)/{\F}_p$ are normal with cyclic Galois group. Moreover, for every positive integer $n$, at most (exactly) one field of the form ${\F}_p(u)$ is an extension of the field ${\F}_p$ of degree $n$. In other words, there is a bijection between the set of all subfields of $\F$ of the kind ${\F}_p(u)$ onto some subset of the set $\mathbb{N}$ consisting of all natural numbers. Concretely, this subset consists of those elements of $\mathbb{N}$ which are equal to the degree $[{\F_p}(u):{\F}_p]$ for some element $u\in \F$. Since any subset of $\mathbb{N}$ is either finite or infinitely countable, our arguments so far allow us to conclude that $\F$ can be presented as a finite or countably infinite union of finite (sub)fields. This union is properly countable uniquely when $\F$ is an infinitely countable field, as pursued.

Next, by what we have already shown, the base field $\F$ is countable, and hence the elements of $\F$ can be linearly ordered as $f_{1}, \ldots, f _{n}, \ldots$~. Set $\F_{n} = \F_p(f_{1}, \ldots, f_{n})$ for every $n\in \mathbb{N}$. It is now easily inspected that $\F$ is equal to the countable union of all its subfields $\F_n$, where $n \geq 1$. Besides, it easily follows that $\F_{n}/\F_p$ is a finite extension, because it is simultaneously an algebraic extension and finitely generated. Likewise, $\F_{n}$ is obviously a subfield of $\F_{n+1}$ for each natural $n\geq 1$. Thus, we arrive at the conclusion that $\F_n$ is a finite field of order $p^{l_n}$, where $l_n=[\F_{n}:\F_p]$. Since $\F_{n}\leq \F_{n+1}$, one deduces that $l_{n+1}$ is divided by $l_n$ and that the equality $l_{n+1}/l_n=[\F _{n+1}:\F _{n}]$ holds for all $n\in \mathbb{N}$, as desired.

As for the second part that such a field is necessarily is now an immediate consequence of the first one.
\end{proof}

The next comments are worthwhile to explain that the conditions in the previous statement cannot be weakened.

\begin{remark} Knowing that $\mathbb{F}_p$ is the simple (finite) field of prime characteristic $p$, routine arguments show that the field $\mathbb{F}_p(X)$ of rational functions of the variable $X$ with coefficients from $\mathbb{F}_p$, which is actually a transcendental extension of $\mathbb{F}_p$, is an example of a countable field of characteristic $p$ which {\it cannot} be constructed as a countable union of finite fields. Moreover, arguing in the same manner, one can see that the field of rational numbers, $\mathbb{Q}$, also cannnot be presented as a countable union of finite fields.
\end{remark}

We now arrive at our first main result in the present paper. Specifically, the following is true:

\begin{thm}\label{main}
Let $n\geq 1$ be an integer and let $\F$ be a field. Then the following two conditions are equivalent:
\begin{enumerate}
  \item Every $n\times n$ companion matrix over $\F$ is the sum of a potent and a square-zero nilpotent over $\F$.
  \item $\F$ is a countable field of positive characteristic, which is an algebraic extension of its minimal simple (finite) subfield, with all subfields of order greater than $n$.
\end{enumerate}
\end{thm}

\begin{proof}
$(1)\implies (2)$. Assume that each $n\times n$ companion matrix $C$ over $\F$ is weakly periodic with nilpotence index $2$. Referring to Proposition \ref{SRnWm}, we have $\F \subseteq\cap_{m \in \mathbb{N}, m \geq 2} W_m$. But, for every $\omega \in W_m$, there exist integer multiples of unity, say $\alpha$ and $u$, such that $(\omega - \alpha)^m=u$ by owing to Lemma \ref{eq4omega}.

Assume in a way of contradiction that $\F$ has zero characteristic. It then follows that each non-zero integer multiple of unity $s$ is invertible. Set $\omega = r \cdot s^{-1}$, and let $r$ and $\omega$ be integers multiple of unity. Hence, $r \cdot s^{-1} \in \cap_{m \in \mathbb{N}, m \geq 2} W_m$. Fix an $m \geq 2$. Thus, there exist integer multiples of unity $\alpha$ and $u$ such that $(r \cdot s^{-1} -\alpha)^m=u$. By the classical Newton's binomial formula, we find that

$$(r \cdot s^{-1})^m+(-1)^1 \binom{m}{1}(r \cdot s^{-1})^{m-1}\alpha+ (-1)^2 \binom{m}{2}(r \cdot s^{-1})^{m-2}\alpha^2+ \ldots$$
$$ + (-1)^{m-1} \binom{m}{m-1}(r \cdot s^{-1})^{1}\alpha^{m-1}+\alpha^m=u$$

\noindent and so

$$r^m+(-1)^1\binom{m}{1}r^{m-1}s\alpha+(-1)^2\binom{m}{2}r^{m-2}s^2\alpha^2+$$
$$\ldots+(-1)^{m-1}r^1s^{m-1}\alpha^{m-1}+\alpha^m \cdot s^m=u \cdot s^m,$$

\noindent which is equivalent to
$$r^m=((-1)^0\binom{m}{1}r^{m-1}\cdot\alpha+(-1)^1\binom{m}{2}r^{m-2}\alpha^2s+$$
$$\ldots+(-1)^{m-2}r\alpha^{m-1}s^{m-2}+\alpha^m\cdot s^{m-1}+us^{m-1})\cdot s.$$

\medskip

\noindent As by assumption the field is of characteristic zero, we may consider with no harm of generality that the above equation is satisfied in $\mathbb{Z}$, whence it follows that $s$ divides $r^m$ for any $s\in \mathbb{Z}^*$, $r \in \mathbb{Z}$, which is manifestly false; for example, by considering $s=2$ and $r=3$. Hence, the characteristic of the field has to be some non-zero prime, say $p>0$. Besides, Proposition \ref{FinSRn} ensures that the field is of necessity countable, because $SR_n$ is countable and $\F\subseteq SR_n$.

Now, assume that the containment $GF(p^l)\subseteq \F$ is fulfilled for some integer $l$ satisfying the inequalities $2^t\neq p^l-1\leq n-1$ for any non-negative integer $t$. Thus, there exists an odd proper divisor $d$ of $p^l-1$. Further, take $q$ a non-negative odd integer such that $q$ and $p^l-1$ are co-prime. Then, $m=d\cdot q$ is odd and $d=gcd(p^l-1,m)$. Let the equation $x^m-1=0$ holds over $\F$. It has $d$ solutions in $GF(p^l)$, namely $1, h, h^2,\ldots, h^{d-1},$ for $h=g^{\frac{p^l-1}{d}},$ where g is a generator of the multiplicative cyclic group of the field $GF(p^l)$. Since $h^d=1,$ we have $$1+h+h^2+\ldots+h^{d-1}=0.$$ Therefore, $$-1=h+h^2+\ldots+h^{d-1}.$$ So, $-1$ is the sum of $d-1$ $m$-th roots of unity with each of them not equal to $1$. Since $d-1<d\leq p^l-1\leq n-1$, we detect that
$$-1=\alpha_1h+\alpha_2h^2+\ldots+\alpha_{d-1}h^{d-1}$$
with $$\alpha_1=\alpha_2=\ldots=\alpha_{d-1}=1,$$
and hence it follows from the proof of Lemma \ref{eq4omega} that
$$(-1-0)^m=\alpha_1^m+\alpha_2^m+\ldots+\alpha_{d-1}^m-0.$$
Also, $(-1)^m=d-1$, and since $m$ is odd, we deduce that $p$ divides $d$. But $d$ divides $p^l-1$. So, by the transitivity law, $p$ must divide $p^l-1$ and thus $p=1$, which is a contradiction.

Therefore, we have $GF(p^l)\subseteq \F$ for some integer $l$ with $2^t= p^l-1\leq n-1$ for some non-negative integer $t$. It follows now that $p-1$ and $r=p^{l-1}+p^{l-2}+\ldots +p+1$ are powers of $2$ and $p-1$ divides $r$. But since $p-1$ divides $r-l$, we receive that $p-1$ divides $l$. So, since $p\neq 2$, we deduce that $l=2u$ with $u$ a positive integer. That is why, $(p^u-1)(p^u+1)=2^t$ and hence $p^u-1$ divides $p^u+1$. Finally, $p^u-1$ divides $2$ and since $p\neq 2$, we obtain $p=3$ and $u=1$.

So, $l=2u=2$ and $GF(p^l)=GF(3^2)$. Let $g$ be a generator of the multiplicative group of $GF(9)$. Since $gcd(3^2-1,12)=4$, there are four $12$-th roots of unity in the field $GF(3^2)$. They are $\epsilon^{0}, \epsilon^6,\epsilon^{a_3}$ and $\epsilon^{a_4}$, where $\epsilon$ is a primitive $12$-th root of unity, whereas $a_3$ and $a_4$ are positive distinct integers less then $12$ and not equal to $6$. However, as $\frac{3^2-1}{4}=2$, we have
$$\{1, -1, \epsilon^{a_3}, \epsilon^{a_4} \}=\{1,g^2,(g^2)^2,(g^2)^3\}.$$
Since $(g^2)^4=1$, it follows that
$$1+g^2+(g^2)^2+(g^2)^3=0.$$ Consequently,
$$-1=-1+ \epsilon^{a_3}+\epsilon^{a_4}.$$
We thus derive two things: $\epsilon^{a_3}=-\epsilon^{a_4}$, and $-1$ is the sum of at most $n$ $12-$th roots of unity, because $3<3^2\leq n$. Therefore, imitating the proof of Lemma \ref{eq4omega}, taking into account that $12$ is even, it follows that
$$(0-(-1))^{12}=(-1)^6\cdot 1^{12}+(-1)^{a_3}\cdot 1^{12}+(-1)^{a_4}\cdot 1^{12}.$$
We thus inspect that the numbers $a_3$ and $a_4$ cannot be simultaneously odd or simultaneously even. In this way, without loss of generality, we can assume that $a_3$ is even and $a_4$ is odd, and since $\epsilon^{a_3}=-\epsilon^{a_4}$, we obtain that
$(-\epsilon)^{a_3}=(-\epsilon)^{a_4}$. However, because $\epsilon$ is not lying in the set $\{-1,0,1\}$, while $a_3$ and $a_4$ are positive integers less then $12$, it follows automatically that $a_3=a_4$, which is a new contradiction.

Now, in order to demonstrate that $\F$ is an algebraic extension of its minimal simple (finite) subfield, we will prove that every non-zero element of $\F$ is a root of unity. In fact, owing to Proposition \ref{FinSRn}, we have that $\F\in SR_n$, the set of all sums of at most $n$ roots of unity. Also, according to Lemma $\ref{SRnWm}$, we get that $SR_n\subseteq \cup_{m\in \mathbb{N},m\geq 2}W_m$, whereas Lemma $\ref{eq4omega}$ enables us that if $x\in W_m$, then there exist $u\in \F_p$ and $\alpha\in \F_p$ such that $(x-\alpha)^m=u$. Likewise, the proofs of the mentioned lemmas tell us that if $x$ is the sum of at most $n$ $m$-th roots of unity within there are exactly $\alpha$ values of $1$, then $(x-\alpha)^m\in\F_p$. So, we may assume that $x=\alpha$. If $x\neq 0$, then $x^{p-1}=1$ and thus $x$ is a root of unity. If, however, $x$ is the sum of at most $n$ roots of unity within there are no values of $1$, then $(x-\alpha)^{m}\in \F_p-\{0\}$. Hence, $x=\alpha+y$ with $\alpha^{p-1}=1$ and $y^{m(p-1)}=1$. If $\alpha\neq 1$, then $x$ is the sum of two roots of unity not equal to $1$. So, $x^{m(p-1)}\in \F_p-\{0\}$ and, therefore, $x^{m(p-1)^2}=1$ Thus, $x$ is a root of unity. Now take $\alpha=1$. But as $$x=1\cdot 1+y=(p-1)\cdot (-1)+y=((n-1)\cdot(-1)+y)+(p-n)\cdot (-1),$$
we have that $(n-1)(-1)+y$ is the sum of $n-1$ roots of unity equal to $(-1)$ and one $m(p-1)-th$ root of unity not equal to $1$. It follows now that $((n-1)\cdot(-1) +y-1\cdot0)^{m(p-1)}=1$. Moreover, if $(p-n)\cdot (-1)=1$, then $p$ divides $n-1$. But $n-1\leq p-1$, so that $p\leq p-1$ which is false. By this contradiction we derive that $x$ is the sum of two $m(p-1)$-th roots of unity not equal to $1$. Consequently, $x^{m(p-1)}\in \F_p-\{0\}$, whence $x^{m(p-1)^2}=1$. Finally, $x$ is a root of unity, as promised. In conclusion, $\F$ is a countable field of positive characteristic, which is an algebraic extension of its minimal simple (finite) subfield, with all subfields of order greater than $n$, as desired.

$(2)\implies (1)$. Let $GF(p^l)$ be contained in $\F$. Take $C_1$ to be an $n\times n$ companion matrix over $GF(p^l)$. Then, as $p^l\geq n+1$, it follows by the application of the main result from \cite{LAMA} that we may decompose $C_1=D_1+N_1$, where $D_1$ is a diagonalizable matrix over $GF(p^l)$ with no multiple eigenvalues and $N_1$ is a nilpotent matrix of nilpotence index $2$ over $GF(p^l)$. Now, Lemma \ref{ADiagIsNonderog} helps us to have that $D_1$ is a non-derogatory potent matrix, and since $\mbox{trace}~N=0$ we can get that $\mbox{trace}~C_1=\mbox{trace}~D_1$. Therefore, any element in $GF(p^l)$ can be the trace of a non-derogatory potent matrix. Hence, each element in $GF(p^l)$ can be the trace of a potent companion matrix.

Now, according to Proposition~\ref{fields}, every countable field of positive characteristic, which is an algebraic extension of its minimal simple (finite) subfield, is a countable union of its finite subfields, it follows at once that every element in $\F$ can be the trace of an $n\times n$ potent companion matrix over $\F$. Thus, letting $C$ be an arbitrary $n\times n$ companion matrix over $\F$. Then, $\mbox{trace}~C\in ST_n$. In conclusion, applying Proposition \ref{wp2compmat}(2), one infers that $C$ can be written as a sum of a potent matrix and a nilpotent matrix of order $2$ over $\F$, as wanted.
\end{proof}

We now come to the case when the potent and square-zero matrices commute each other.

\subsection{The commuting case}

We will attack here the commuting case of weakly periodic with nilpotence index $2$ companion matrices, that is, the existing potent and square-zero nilpotent matrices commute each other. To that aim, we first need the following two technical conventions.

\begin{lemma}\label{rootsOfUnityLem}
Let $n$ be a positive integer and $\F$ a field. If, for every $n\times n$ companion matrix $C$ over $\F$, there exist an integer $t>1$, a potent matrix $P$ such that $P=P^t$ and a square-zero nilpotent $N$ such that $C=P+N$ with $PN=NP$, then the following two points are true:
\begin{enumerate}
\item If $C$ is invertible, then $\chi_C$ divides $(X^{t-1}-1)^2$.
\item If $\lambda_1,\lambda_2,\dots,\lambda_n$ are non-zero elements of $\F$, then there exists an integer $t>1$ such that $\lambda_i^{t-1}=1$ for every $i\in \{1,2,\dots,n\}$.
\end{enumerate}
\end{lemma}

\begin{proof}
(1) Since $C=P+N$ and $PN=NP$, it follows that $CP=PC$ and $CN=NC$. But we have $(C-N)^t=C-N$. Since $CN=NC$, we can apply Newton's binomial formula to argue that there exists an $n\times n$ matrix over $\F$ such that
$$C^t-tCN+MN^2=C-N.$$ But as $N^2=0$, we obtain that
$$C(C^{t-1}-tN)=C-N.$$
Multiplying with $C^{t-1}+tN$, we get that
$$C(C^{2t-2}-t^2N^2)=C^t+tCN-C^{t-1}N+tN^2,$$ and using that $N^2=0$, we write the equality
$$C^{2t-1}=C^t+(tC-C^{t-1})N.$$ But $C$ is invertible, and so
$$C^{2t-2}-C^{t-1}=(tId_n-C^{t-2})N.$$ Now, bearing in mind that $CN=NC$ and $N^2=0$, we infer
$$((C^{t-1})^2-C^{t-1})^2=0,$$ and since $C$ is invertible, we conclude
$$(C^{t-1}-1)^2=0.$$ Therefore, the minimal polynomial of $C$ divides $(X^{t-1}-1)^2$. But the minimal polynomial of $C$ is the characteristic polynomial, say $\chi_C$, of $C$. Summarizing all the information so far, $\chi_C$ divides $(X^{t-1}-1)^2$, as promised.

(2) Just take $C$ to be the companion matrix with eigenvalues $\lambda_1,\lambda_2,\dots,\lambda_n$ and apply the preceding point (1) along with the fact that $\lambda_1,\lambda_2,\dots,\lambda_n$ are from $\F$.
\end{proof}

\begin{lemma}\label{fixedPointsLem}
Let $n$ be a positive integer and $\F$ a field.
If, for every $n\times n$ companion matrix $C$ over $\F$, there exist an integer $t>1$, a potent matrix $P$ such that $P=P^t$ and a square zero nilpotent $N$ such that $C=P+N$ with $PN=NP$, then the following two conditions are valid:
\begin{enumerate}
\item If $C$ is invertible, then there exists a polynomial $q\in \F[X]$ such that $\chi_C$ divides $(q(X)-X)^2$.
\item If $\lambda_1,\lambda_2,\dots,\lambda_n$ are non-zero elements of $\F$, then there exists a polynomial $q\in \F[X]$ such that $q(\lambda_i)=\lambda_i$ for every $i\in \{1,2,\dots,n\}$.
\end{enumerate}
\end{lemma}

\begin{proof}
(1) Since $C=P+N$ and $PN=NP$, it follows that $CP=PC$ and $CN=NC$. It is well known that all matrices that commute with a companion matrix $C$ can be interpreted just as polynomials in $C$ over $\F$. So, there exists $q\in \F[X]$ with $C=q(C)+N$. Now, since $N^2=0$, it follows that $(q(C)-C)^2=0$. Therefore, the minimal polynomial of $C$ obviously divides $(q(X)-X)^2$. But the minimal polynomial of $C$ is the characteristic polynomial, say $\chi_C$, of $C$. Summarizing all the information thus far, $\chi_C$ divides $(q(X)-X)^2$, as asked for.

(2) Just take $C$ to be the companion matrix with eigenvalues $\lambda_1,\lambda_2,\dots,\lambda_n$ and employ the preceding condition (1) together with the fact that $\lambda_1,\lambda_2,\dots,\lambda_n$ are from $\F$.
\end{proof}

The following claim from number theory is a well-known folklore fact, but we state it here only for the sake of completeness and the reader's convenience.

\begin{lemma}\label{gcd}
Let $a,b,c$ be three positive integers. Then, $gcd(bc,a)$ divides $gcd(b,a)\cdot gcd(c,a)$.
\end{lemma}

\begin{proof}
Let we set $f=gcd(a,bc)$, $f_1=gcd(a,b)$ and $f_2=gcd(a,c)$. Then, there exist integers $s_1,s_2,t_1,t_2$ such that
$$f_1=s_1a+t_1b$$
and
$$f_2=s_2a+t_2c.$$
Thus,
$$f_1f_2=s_1s_2a^2+s_1t_2ac+t_1s_2ba+t_1t_2bc.$$
But $f\mbox{ divides }a$ and $f\mbox{ divides }bc$, so $f\mbox{ divides }a^2$, $f\mbox{ divides }ac$, $f\mbox{ divides }ba$ and $f\mbox{ divides }bc$. Hence, $f\mbox{ divides }f_1f_2$ whence $gcd(bc,a)$ divides $gcd(b,a)\cdot gcd(c,a)$, as formulated.
\end{proof}

The next comments are needed to explain the complicated situation in the commuting case.

\begin{remark}\label{remCommuting}
Let $\lambda_1, \lambda_2,\dots, \lambda_n$ be distinct non-zero elements of $\F$. By Lemma \ref{fixedPointsLem}(2), there exists $q\in \F[X]$ such that $q(\lambda_i)=\lambda_i$ for every $i\in \{1,2,\dots,n\}$. Take $GF(p^{l_1})$ such that $\lambda_1, \lambda_2,\dots, \lambda_n$ are in $GF(p^{l_1}),$ $GF(p^{l_2})$ and such that $q\in GF(p^{l_2})[X]$ with $l=max(l_1,l_2)$. Then, $\lambda_1, \lambda_2,\dots, \lambda_n$ are in $GF(p^{l})$, while $q\in GF(p^{l})[X]$.

Furthermore, since $\lambda_1, \lambda_2,\dots, \lambda_n$ are in $\F$, it follows from Lemma \ref{rootsOfUnityLem}(2) that there exists a non-negative integer $t>1$ such that $\lambda_i^{t-1}=1$ for every $i\in \{1,2,\dots,n\}$. Now, since $\lambda_1, \lambda_2,\dots, \lambda_n$ are in $GF(p^l)$, we can infer that $$\{\lambda_1, \lambda_2,\dots, \lambda_n\}\subseteq \{h,h^2,\dots, h^d=1\},$$ where $h=g^{\frac{p^l-1}{d}}$ with $g$ a generator of the multiplicative group of $GF(p^l)$, and $d=gcd(p^{l}-1,t-1)$. Therefore, $n\leq d$ and there exist $\{j_1,j_2,\dots,j_n\}\subseteq\{1,2,\dots,d\}$ such that $\lambda_i=h^{j_i}$. Finally, we extract that $q(h^{j_i})=h^{j_i}$ for every $i\in \{1,2,\dots,n\}$ with $h^d=1$, as expected.
\end{remark}

The next example illustrates some concrete aspects of our calculating manipulations.

\begin{example}
If in the previous Remark~\ref{remCommuting} we take $\lambda_1=g^i$ and $\lambda_2=g^{i+1}$ with $i\in \{1,2,\dots,p^l-2\}$, then one inspects that
$$\{g^i,g^{i+1}\}\subseteq \{g^{\frac{p^l-1}{d}},g^{2\frac{p^l-1}{d}},\dots,g^{d\frac{p^l-1}{d}}\}.$$
Therefore,
$$\{i,i+1\}\subseteq \{ \frac{p^l-1}{d},2\frac{p^l-1}{d},\dots,d\frac{p^l-1}{d}\}. $$
Hence, the ordinary fraction $\frac{p^l-1}{d}$ is a common divisor of both $i$ and $i+1,$ so it is necessarily $1$. We thus have now that $p^l-1=d=gcd(p^l-1,t-1)$. In conclusion, $p^l-1$ divides $t-1$, as desired to demonstrate.
\end{example}

We are now ready to establish our second main result.

\begin{thm} Suppose $n\geq 1$ is an integer and $\F$ is a field. If all companion $n\times n$ matrices over $\F$ are the sum of a potent matrix and a square-zero matrix over $\F$ which matrices commute each other, then $\F$ is a finite field of order greater than $n$.
\end{thm}

\begin{proof}
Suppose $n\geq 1$ is an integer, $\F$ is a field and all companion $n\times n$ matrices over $\F$ are the sum of a potent matrix and a square-zero matrix over $\F$ which commute. Then, by Theorem \ref{main}, we have that $\F$ is a countable field of positive characteristic (which is also an algebraic extension of a finite field). Assume now the contrary that $\F$ is infinite. Thus, the second part in the statement of Proposition~\ref{fields} applies to write that $\F=\cup_{i=1}^\infty GF(p^{l_i})$, where $l_i$ divides $l_{i+1}$ for every positive integer $i$. Therefore, there exists the sequence of integers greater than $1$, say $(c_i)_{i\geq 1}$ such that $l_{i+1}=l_i\cdot c_i$ and $(l_i)_{i\geq 1}$ is a strictly increasing infinite sequence of positive integers.

Take $g_1$ to be the generator of the multiplicative group of $GF(p^{l_1})$ and put $$\lambda_1=g_1^{[\frac{p^{l_1}-1}{n}]}, \lambda_2=g_1^{2\cdot [\frac{p^{l_1}-1}{n}]},\dots ,
\lambda_n=g_1^{n\cdot [\frac{p^{l_1}-1}{n}]}.$$
Consequently, Lemma~\ref{rootsOfUnityLem} allows us to have the existence of a positive integer $t>1$ such that $$\lambda_1^{t-1}=\lambda_2^{t-1}=\dots=\lambda_n^{t-1}=1.$$
Let $C$ be the companion matrix with eigenvalues $\lambda_1,\lambda_2,\dots, \lambda_n$. Note that the integer $t$ above has the property that $C=P+N$ with $P^t=P$, $N^2=0$ and $PN=NP$. If, eventually, there are more than one such decompositions of $C$, then we can take $t$ to be the minimum of such integers $t$'s. So, $t$ can be chosen to be a fixed uniquely determined integer having the mentioned property.

Furthermore, since we are working in $GF(p^{l_1})$, Remark \ref{remCommuting} leads us to the relation
$$\{g_1^{[\frac{p^{l_1}-1}{n}]}, g_1^{2\cdot [\frac{p^{l_1}-1}{n}]},\dots ,
g_1^{n\cdot [\frac{p^{l_1}-1}{n}]}\}\subseteq \{g_1^{\frac{p^{l_1}-1}{d_1}},g_1^{2\frac{p^{l_1}-1}{d_1}},\dots, g_1^{d_1\frac{p^{l_1}-1}{d_1}}\},$$
where $d_i=gcd(p^{l_i}-1,t-1)$. Take $j_i\in \{1,2,\dots,d_i\}$ such that $$g_1^{[\frac{p^{l_1}-1}{n}]}=g_1^{j_1\cdot\frac{p^{l_1}-1}{d_1}},$$ and since the order of $g_1$ in the  multiplicative group of $GF(p^{l_1})$ is $p^{l_1}-1$, it follows that
$$[\frac{p^{l_1}-1}{n}]=j_1\cdot\frac{p^{l_1}-1}{d_1}.$$
Analogously, we obtain that $$[\frac{p^{l_1}-1}{n}]=j_i\cdot\frac{p^{l_i}-1}{d_i},$$ for any positive integer $i$.
Therefore, $$j_i\cdot \frac{p^{l_i}-1}{d_i}=j_{i+1}\cdot \frac{p^{l_{i+1}}-1}{d_{i+1},}$$
and so $$\frac{j_i}{j_{i+1}}=\frac{d_i}{d_{i+1}}\cdot \frac{(p^{l_i})^{c_i}-1}{p^{l_i}-1}.$$
Take $s_i=(p^{l_i})^{c_i-1}+(p^{l_i})^{c_i-2}+\dots+p^{l_i}+1.$ It thus follows that
$$\frac{j_i}{j_{i+1}}=\frac{gcd(p^{l_i}-1,t-1)}{gcd((p^{l_i})^{c_i}-1,t-1)}\cdot s_i.$$
Also, by Lemma \ref{gcd} we have that there exists a strictly positive integer $k_i$ such that
$$gcd((p^{l_i})^{c_i}-1,t-1)=\frac{gcd(p^{l_i}-1,t-1)\cdot gcd(s_i,t-1))}{k_i}.$$
Now, we obtain that
$$\frac{j_i}{j_{i+1}}=k_i\cdot \frac{s_i}{gcd(s_i,t-1)}.$$

\medskip

However, since $gcd(s_i,t-1)$ divides $s_i$, it follows that $j_{i+1}$ divides $j_i$ for every positive integer $i\geq 1$. But $j_i\geq 1$ and thus there exists $i_0\geq 1$ such that $j_i=j_{i+1}$ for every $i\geq i_0$. So, we obtain
$$k_i\cdot \frac{s_i}{gcd(s_i,t-1)}=1,$$ which forces $s_i=gcd(s_i,t-1)$.
Hence, $s_i$ divides $t-1$ and so $s_i\leq t-1$. But $p^{l_i}<s_i$ and then $p^{l_i}<t-1$, for every $i\geq i_0$,
which is in sharp contradiction with the fact that $(l_i)_{i\geq 1}$ is a strictly increasing infinite sequence of positive integers. In conclusion, one has that $\F$ is finite, as claimed. We next once again apply Theorem~\ref{main} to get that the order of the field is greater than $n$, as asserted.
\end{proof}

In the spirit of the last result, we pose the following conjecture.

\medskip

\noindent{\bf Conjecture:} Given $n\geq 1$ is an integer and $\F$ is a field. Then all companion $n\times n$ matrices over $\F$ are the sum of a potent matrix and a square-zero matrix over $\F$ which matrices commute each other if, and only if, $\F$ is a finite (and hence potent) field and $n=1$.

\medskip

On the other side, in regard to \cite{BCDM} and \cite{KLZ}, we close our work with the following question of some interest.

\begin{prob} Suppose that $D$ is a division ring and $n\geq 1$ is an integer. Does it follow that the matrix ring $M_n(D)$ is weakly periodic if, and only if, $D$ is a finite (and hence potent) field?
\end{prob}

\medskip

\noindent{\bf Funding:} The scientific work of Peter V. Danchev was partially supported by the Bulgarian National Science Fund under Grant KP-06 No 32/1 of December 07, 2019 and by the Junta de Andaluc\'ia, FQM 264.

\medskip

%\noindent{\bf Acnowledgement.} The first-named author, Peter V. Danchev, would like to express his gratitude to Prof. Ivan D. Chipchakov for the constructive private communication on field theory, which led to writing of Proposition~\ref{fields}.

\end{document}